\documentclass[11pt]{article} 

\usepackage{amssymb,amsmath,amsthm}
\usepackage[letterpaper,hmargin=1.0in,vmargin=1.0in]{geometry}
\usepackage{graphics} 
\newtheorem{theorem}{Theorem}[section]
\newtheorem{lemma}[theorem]{Lemma}
\newtheorem{proposition}[theorem]{Proposition}
\newtheorem{corollary}[theorem]{Corollary}

 \newcommand{\N}{\mathcal N}

 \newcommand{\gex}{G_{ex}}

\begin{document} 
\title{$H$-free subgraphs of dense graphs maximizing the number of
cliques and their blow-ups} 
\author{
Noga Alon\thanks{Sackler School of
Mathematics and Blavatnik School of Computer
Science, Tel Aviv University, Tel Aviv 69978,
Israel.
Email: nogaa@tau.ac.il.
Research supported in part by a BSF grant,
an ISF grant and a GIF grant.
}\and
Clara Shikhelman \thanks{Sackler School of Mathematics,
Tel Aviv University,
Tel Aviv 69978, Israel.
Email: clarashk@post.tau.ac.il.
Research supported in part by an ISF grant.
}
}
\setlength{\parskip}{1ex plus 0.5ex minus 0.2ex}

\maketitle 
\begin{abstract} 
We consider the structure of $H$-free
subgraphs of graphs with high minimal degree. We prove that
for every $k>m$ there exists an $\epsilon:=\epsilon(k,m)>0$ so that the
following holds. For every graph $H$ with chromatic number $k$ from
which one can delete an edge and reduce the chromatic number, and for
every graph $G$ on $n>n_0(H)$ vertices in which all degrees are at least
$(1-\epsilon)n$, any subgraph of $G$ which is $H$-free and contains the
maximum number of copies of the complete graph $K_m$ is
$(k-1)$-colorable.

We also consider several extensions for the case of a general forbidden
graph $H$ of a given chromatic number, and for subgraphs maximizing the
number of copies of balanced blowups of complete graphs. \end{abstract}

\section{Introduction}

The well known theorem of Tur\'an (\cite{T}) 
states that a $K_k$-free subgraph of the complete graph on $n$ vertices
with the maximum possible number of edges is $k-1$-chromatic. Erd\H{o}s,
Stone and Simonovits show in \cite{ESt}, \cite{ES66}
that for
general $H$ with $\chi(H)=k$ the maximum possible number of edges in an
$H$-free graph on $n$ vertices is at most $o(n^2)$ more than the number
of edges in a $k-1$-chromatic graph on $n$ vertices. In \cite{ASh} it is
shown that the same holds for $H$-free subgraphs of the complete graph
that have the maximum possible number of copies of $K_m$ for a fixed $m$
such that $k>m\geq 2$.

Looking at subgraphs of general graphs $G$ it is clear that a $K_m$-free
subgraph of $G$ with the maximum possible number of edges has at least as
many edges as the largest $m-1$-partite subgraph. In  \cite{E83}
Erd\H{o}s asked for which graphs there is an equality between the
two. In \cite{Al} it is shown that this is the case for line graphs
of bipartite graphs. In a different direction, in
\cite{BSTT} it is proved that if a graph has a high enough minimum
degree then any subgraph of it which is $K_3$-free and has the maximum
possible number of edges is bipartite. In \cite{BKS} a stronger
bound is given on the minimum degree ensuring this. 
Before stating a generalization
of these theorems we introduce some notation.

For a graph $G$, fixed graphs $H$ and $T$ and an integer $k$ let
$G_{part(k),T}$ be a $k$-partite subgraph of $G$ with the maximum possible
number of copies of $T$ and let $\mathcal{G}_{ex}(T,H)$ be the family of
subgraphs of $G$ that are $H$-free and have the maximum possible number
of copies of $T$. Let $\N(G,T)$ denote the number of copies of $T$ in $G$.
Call a graph $H$ edge critical if there is an edge $\{u,v\}\in E(H)$
whose removal reduces the chromatic number of $H$.

In \cite{AShS} the following theorem is proved, generalizing the results
in \cite{BSTT} and \cite{BKS}. Throughout the paper we 
denote by $\delta(G)$ 
the minimum degree in the graph $G$.
\begin{theorem}[\cite{AShS}] 
Let $H$ be
a graph with $\chi(H)=k+1$. Then there are positive constants
$\gamma:=\gamma(H)$ and $\mu:=\mu(H)$	such that if $G$ is a graph on
$n>n_0(H)$ vertices with $\delta(G)>(1-\mu)n$ then for every $G_{ex}\in
\mathcal{G}_{ex}(K_2,H)$ \begin{enumerate} \item If $H$ is edge critical
then $\N(G_{part(k),K_2},K_2) = \N(G_{ex},K_2)$ \item Otherwise,
$\N(G_{part(k),K_2},K_2) \leq \N(G_{ex},K_2) \leq
\N(G_{part(k),K_2},K_2)+O(n^{2-\gamma}).$ \end{enumerate}
\end{theorem} 

In the present
short paper we prove two theorems for $H$-free subgraphs assuming $H$
is edge critical. The first is for subgraphs maximizing the number of
copies of $K_m$ and the second for subgraphs maximizing the number of
blow-ups of $K_m$. We also establish a proposition concerning graphs $H$
that are not edge critical. 

\begin{theorem} 
\label{thm:mainComplete} 
For every two integers $k>m$
and every edge critical graph $H$ such that $\chi(H)=k$ there exist
constants $\epsilon:=\epsilon(k,m)>0$ and $n_0=n_0(H)$ such that the
following holds. Let $G$ be a graph on $n>n_0$ vertices with
$\delta(G)\geq (1-\epsilon)n$, then for every $G_{ex}\in
\mathcal{G}_{ex}(K_m,H)$ the graph $G_{ex}$ is $(k-1)$-colorable.
\end{theorem}

For integers $m$ and $t$ let $K_m(t)$ denote the $t$-blow-up of $K_m$,
that is, the graph obtained by replacing each vertex of $K_m$ by an
independent set of size $t$ and each edge by a complete bipartite
graph between the corresponding independent sets.

\begin{theorem} 
\label{thm:mainBlowUp} 
For integers $m$ and $t$ and
every edge critical $H$ such that $\chi(H)=m+1$ there exist constants
$\epsilon:=\epsilon(m,t)$ and $n_0:=n_0(H)$ such that the following
holds. Let $G$ be a graph on $n>n_0$ vertices with
$\delta(G)>(1-\epsilon)n$, then every $G_{ex}\in
\mathcal{G}_{ex}(K_m(t),H)$ is $m$-colorable.  
\end{theorem}

Finally, for graphs $H$ which are not edge critical we prove the
following. 
\begin{proposition}
\label{prop:otherH} 
For every integers
$m<k$ and $t$ and graph $H$ such that $\chi(H)=k$ there exists
$\epsilon:=\epsilon(m,t,k)$ and $n_0:=n_0(H)$ such that the following
holds. Let $G$ be a graph on $n>n_0$ vertices with
$\delta(G)>(1-\epsilon)n$ and assume that $t=1$ or $k=m+1$, then every
$G_{ex}\in \mathcal{G}_{ex}(K_m(t),H)$ can be made $k-1$
colorable
by deleting $o(n^2)$ edges. 
\end{proposition}

Theorems \ref{thm:mainComplete} and \ref{thm:mainBlowUp} cannot be
directly generalized to graphs $H$ that are not edge critical as we can
add to any $k-1$-partite graph an edge without creating a copy of such
$H$. On the other hand, we believe that the error term $o(n^2)$ in
Proposition \ref{prop:otherH} can be improved to
$O(n^{2-\delta})$ for some $\delta:=\delta(H)$. 

The rest of this short paper is organized  as follows. In Section 2
we state several known results and prove some helpful lemmas. Section
3 contains the proof of of Theorem  \ref{thm:mainComplete}. Theorem
\ref{thm:mainBlowUp} is proved in Section 4 and the proof of
Proposition \ref{prop:otherH} appears in Section 5. The final Section
6 contains some concluding remarks and open problems.

\section{Preliminary results}

We start by stating several results about $H$-free graphs with high
degrees
and by deducing a corollary. Some of the theorems stated 
are simplified versions of the original results.

The first result about $K_k$-free graphs is by  Andr\'asfai, Erd\H{o}s
and S\'os.

\begin{theorem}[\cite{AES}] 
\label{thm:highDegforKk} 
Let $G$ be a graph
on $n$ vertice. If $G$ is $K_k$-free and $\delta(G)\geq
(1-\frac{3}{3k-4})n$ then $\chi(G)\leq k-1$. 
\end{theorem}

A generalization of Theorem \ref{thm:highDegforKk} proved in
\cite{ES} is the following.

\begin{theorem}[\cite{ES}]\label{thm:highDegforH} Let $H$ be a fixed
edge critical graph which is not $K_k$ and assume $\chi(H)=k$. If $G$ is
a graph on $n>n_0(H)$ vertices which is $H$-free and contains a copy of
$K_k$ then $\delta(G)\leq (1-\frac{1}{k-3/2})n+O(1)$. \end{theorem}

This implies that if 
$n$ is large enough, $\delta(G)\geq (1-\frac{3}{3k-4})n\geq
(1-\frac{1}{k-3/2})n+O(1)$, and if $G$ is $H$-free for some edge critical
graph $H$ with $\chi(H)=k$ then it must also be $K_k$-free. Together
with Theorem \ref{thm:highDegforKk} we get the following corollary:

\begin{corollary}\label{cor:ifFreeThenChrom} Let $H$ be a fixed edge
critical graph such that $\chi(H)=k$. Let $G$ be a graph on $n>n_0(H)$
vertices which is $H$-free and satisfies 
$\delta(G)\geq (1-\frac{3}{3k-4})n$, then
$\chi(G)\leq k-1$. \end{corollary}

We next state the graph removal lemma as it appears in \cite{CF} (see
also \cite{AFKS}, \cite{RSz} and \cite{Fo}) and prove a simple lemma
using it. Throughout the paper we denote by $v(G)$ the number of vertices 
in the graph $G$.

\begin{theorem}[The graph removal lemma]\label{lem:graphRemovalLemma} For
any graph H with $v(H)$ vertices 
and any $\epsilon > 0$, there exists a $\delta > 0$ such that
any graph on $n$ vertices which contains at most $\delta n^{v(H)}$
copies of $H$ can be made $H$-free by removing at most $\epsilon n^2$
edges.
\end{theorem}

Throughout the paper, for fixed graphs $T$ and $H$ and an integer $n$ we
denote by $ex(n,T,H)$ the maximum possible number of copies of $T$ in an
$H$-free graph on $n$ vertices. 

\begin{lemma}\label{lem:HfreeIsLikeK_k} Let $H$ be a fixed graph such
that $\chi(H)=k$ and let $G$ be an $H$-free graph on $n$ vertices,
where 
$n>n_0(H)$. Then $G$ can be made $K_k$-free by deleting $o(n^2)$ edges.
\end{lemma}

\begin{proof} Note, first, that the number of copies of
$K_k$ in $G$ is $o(n^k)$.  Indeed, in 
\cite{ASh} it is shown that if a graph $H$
is a subgraph of a blow-up of a graph $T$ then 
$ex(n,T,H)=o(n^{v(T)})$.

Since $\chi(H)=k$, $H$ is contained in a blow-up of $K_k$ and
hence $\N(G,K_k)\leq ex(n,K_k,H)\leq o(n^{k})$. By the
graph removal lemma $G$ can be made $K_k$-free by
removing $o(n^2)$ edges, as needed. \end{proof}

We next prove two additional more technical lemmas.

\begin{lemma} \label{lem:counting} Let $G$ be a graph on $n$ vertices
satisfying $\delta(G)>(1-\epsilon)n$ for some fixed $\epsilon>0$, 
and let $m<k$ and $t$ be integers. 
\begin{enumerate} 
\item 
$\N(G_{part(k-1),K_m},K_m)\geq(1+o(1))
(1-\frac{m(m-1)}{2}\epsilon)\binom{k-1}{m}(\frac{n}{k-1})^{m}$ \item Let
$k=m+1$ then $\N(G_{part(m),K_m(t)},K_m(t))\geq(1+o(1))(1-c'_2\epsilon)
n^{mt} \frac{1}{(t!m^t)^{m}}$ \end{enumerate} where $c'_2:=c'_2(m,t)$.
\end{lemma} 
\begin{proof} 
To prove part 1 note that as $\delta(G)\geq
(1-\epsilon)n$ the number of copies of $K_m$ in $G$ is at least
$$
n\cdot (1-\epsilon)n\cdot (1-2\epsilon)n\dots
(1-(m-1)\epsilon)n\frac{1}{m!}\geq(1+o(1))
n^m(1-\frac{m(m-1)}{2}\epsilon)\frac{1}{m!} 
$$

Randomly partitioning the graph into $k-1$ sets yields a graph in which
the expected number of copies of $K_m$ is a least: 
\begin{align*} &
(1+o(1))n^m(1-\frac{m(m-1)}{2}\epsilon)\frac{1}{m!}\cdot
\frac{k-2}{k-1}\cdot \frac{k-3}{k-1}\dots \frac{k-m}{k-1} \\=&
(1+o(1))n^m(1-\frac{m(m-1)}{2}\epsilon)\frac{(k-2)!}
{m!(k-1)^{m-1}(k-(m+1))!} 
\\=&(1+o(1))(1-\frac{m(m-1)}{2}\epsilon)\binom{k-1}{m}
(\frac{n}{k-1})^{m}. 
\end{align*} 
Thus $G_{part(k-1),K_m}$ should have at 
least as many copies. This proves (1).

Similarly to prove part 2 observe that 
the number of copies of $K_m(t)$ in $G$ is at least 
$$ 
\frac{1}{m!(t!)^m}(n)_t((1-t\epsilon)n)_t\cdot  \ldots \cdot
((1-(m-1)t\epsilon)n)_t \geq
(1+o(1))\frac{n^{mt}}{m!(t!)^m}(1-c'_2\epsilon). 
$$

Randomly partitioning $G$ into $m$ parts gives a graph in which the
expected number of copies of $K_m(t)$ is at least \begin{align*} &
(1+o(1))(1-c'_2\epsilon)\frac{n^{mt}}{m!(t!)^m}
(\frac{1}{m^{t-1}})(\frac{m-1}{m}
\frac{1}{m^{t-1}})\dots (\frac{1}{m}\frac{1}{m^{t-1}})
\nonumber \\ 
=&(1+o(1))(1-c'_2\epsilon)n^{mt} \frac{1}{m!(t!)^m} 
\frac{m!}{m^{mt}} \nonumber \\ =&(1+o(1))(1-c'_2\epsilon)n^{mt} 
\frac{1}{(t!m^t)^{m}} 
\end{align*} 
and thus $G_{part(m),K_m}$ must have at least as many
copies of $K_m(t)$. \end{proof}


\begin{lemma} 
\label{lem:returningVertices} Let $G$ be a graph on $n$
vertices with $\delta(G)>(1-\epsilon)n$ for some fixed $\epsilon>0$ and
let $t$ and $m<k$ be integers. For a set $U\subseteq V(G)$
satisfying $|U|\geq \alpha n$ for some fixed $\alpha>0$ let $f_{k,m,t}(U)$ 
be the maximum number of
copies of $K_m(t)$ in a $k-1$-partite subgraph of $G[U]$. Then there
exist constants $c_1:=c_1(k,m,\alpha)$ and $c_2:=c_2(m,t,\alpha)$
such that for every  $v\in V(G)\setminus U$ 
\begin{enumerate} 
\item 		$
f_{m,k,1}(U\cup\{v\})\geq
f_{m,k,1}(U)+(1+o(1))|U|^{m-1}\binom{k-1}{m}
\frac{m}{(k-1)^{m}}(1-c_1\epsilon) $ \label{lem:Caes1} 
\item 		
$ f_{m,m+1,t}(U\cup\{v\})\geq f_{m,m+1,t}(U)
+(1+o(1))|U|^{mt-1}\frac{mt}{(t!m^t)^m}(1-c_2\epsilon) 
$\label{lem:Caes2} 
\end{enumerate}
\end{lemma}

\begin{proof} 
Let $|U|=q\geq \alpha n$. We first prove part \ref{lem:Caes1}. 
Fix a partition of $G[U]$ into $k-1$ parts with $f_{m,k,1}(U)$
copies of $K_m$. By
Lemma \ref{lem:counting}, part 1 
the number of copies of $K_m$ is at least: 
$$
(1+o(1))q^m(1-\frac{m(m-1)}{2}\epsilon)\binom{k-1}{m}\frac{1}{(k-1)^m}.
$$ 
Averaging we get that there is a vertex, say $w\in U$, so that the
number of copies of $K_m$ it takes part in is at least $$
(1+o(1))\frac{m}{q}q^m(1-\frac{m(m-1)}{2}\epsilon)\binom{k-1}{m}
\frac{1}{(k-1)^m}= (1+o(1))q^{m-1}(1-\frac{m(m-1)}{2}
\epsilon)\binom{k-1}{m}\frac{m}{(k-1)^m} $$

Let $U_1,\dots,U_{k-1}$ be the above fixed partition of $U$ which
has
$f_{m,k,1}(U)$ copies of $K_m$ and assume, without loss of
generality, that $w\in U_{k-1}$. We
add $v$ to $U_{k-1}$ and bound from below the number of copies of $K_m$
we add by doing this. Let $b_i=|U_i|$ and let $d_i$ be the number of
neighbors $w$ has in $U_i$ which are not neighbors of $v$. Note that
$\sum_{i\in[k-2]} d_i \leq \epsilon n$ and $\sum_{i\in[k-2]} b_i =q$.

For each $U_i$ we estimate the number of copies of $K_m$ in which $w$ takes
part that use vertices from $U_i$ that are not neighbors of $v$.
There are $d_i$ of those, and in the worst case each such vertex is
connected to all of the sets $U_j$ for $j\ne i,k-1$. Thus the number of
copies of $K_m$ that $w$ takes part in and $v$ does not is at most
\begin{align*} &\sum_{i=1}^{k-2}d_i
\big(\sum_{\{j_1,\dots,j_{m-2}\}\subseteq [k-2]\setminus i}b_{j_1}\dots
b_{j_{m-2}}\big)\\ 
\leq &(d_1+\dots +d_{k-2})(b_1+\dots+b_{k-2})^{m-2}\frac{1}{(m-2)!}\\
\leq &\epsilon n\cdot q^{m-2}\leq \frac{1}{\alpha}\epsilon q^{m-1} 
\end{align*}

And so by adding $v$ the number of copies of $K_m$ added is at least

\begin{align}
&(1+o(1))q^{m-1}[(1-\frac{m(m-1)}{2}\epsilon)\binom{k-1}{m}
\frac{1}{(k-1)^{m}}-\frac{\epsilon}{\alpha}] \nonumber\\ \nonumber 
=&(1+o(1))q^{m-1}\binom{k-1}{m}\frac{1}{(k-1)^{m}}(1-c_1\epsilon) 
\end{align}

The proof of (\ref{lem:Caes2}) is similar. By Lemma
\ref{lem:counting} part 2, in any partition of $G[U]$ into $m$ parts in
which the number of copies of $K_m(t)$ is $f_{m,m+1,t}(U)$, 
this number is at least
$$ 
(1+o(1))\frac{1}{(t!m^t)^{m}}(1-c'_2\epsilon)q^{mt}. 
$$

Let $U=U_1\cup...\cup U_m$ be such a partition. By averaging there must
be a vertex, say $w\in U$, such that the number of copies of $K_m(t)$ it
takes part in is at least: $$ (1+o(1))\frac{mt}{q}q^{mt}
\frac{1}{(t!m^t)^{m}}(1-c'_2\epsilon)=
(1+o(1))q^{mt-1}\frac{mt}{(t!m^t)^{m}}(1-c'_2\epsilon) $$

Assume, without loss of generality,
that $w\in U_m$, and let us add $v$ to $U_m$. Let
$b_i=|U_i|$ and let $d_i$ be the vertices in $U_i$ that are neighbors of
$w$ and not of $v$. Then the number of copies of $K_m(t)$ in this
partition that $w$ takes part in and $v$ does not is at most
\begin{align*} 
&\binom{b_m}{t-1}\sum_{i=1}^{m-1} 
d_i\binom{b_i}{t-1}\prod_{j\in [m-1]\setminus i}\binom{b_j}{t} \\
<&(\sum_{i=1}^{m-1} d_i)\binom{q}{mt-2} <c''_2\epsilon q^{mt-1}
\end{align*} where the last inequality is true for some
$c''_2:=c''_2(m,t,\alpha)$. Thus when adding $v$ to $U_m$ the number of copies
of $K_m(t)$ added is at least $$
q^{mt-1}[\frac{mt}{(t!m^t)^{m}}(1-c'_2\epsilon)-c''_2\epsilon]
=q^{mt-1}\frac{mt}{(t!m^t)^{m}}(1-c_2\epsilon) 
$$ 
as needed. \end{proof}

\section{Maximizing the number of cliques}

In the proof of Theorem \ref{thm:mainComplete} we use the following result
from \cite{ASh}. \begin{proposition}[\cite{ASh}]\label{prop:ex(n,K_m,H)}
Let $H$ be a graph such that $\chi(H)=k>m$ then
$$ex(n,K_m,H)=(1+o(1))\binom{k-1}{m}(\frac{n}{k-1})^m$$
\end{proposition}

\begin{proof}[Proof of Theorem \ref{thm:mainComplete}] 
Let $G$ and $H$ be 
as in the theorem and let $\gex\in \mathcal{G}_{ex}(K_m,H)$.  	If
$\delta(\gex)\geq (1-\frac{3}{3k-4})n$, as $H$ is edge-critical, by
Corollary \ref{cor:ifFreeThenChrom} $\chi(\gex)\leq k-1$ and we are
done. Thus assume towards contradiction that $\delta(\gex)<
(1-\frac{3}{3k-4})n$.

As any partition of $G$ into $k-1$ parts is $H$-free, by Lemma
\ref{lem:counting}, part 1, 
the number of copies of $K_m$ in $G_{ex}$  must be
at least
\begin{equation}
\label{e31}
(1+o(1))(1-\frac{m(m-1)}{2}\epsilon)\binom{k-1}{m}(\frac{n}{k-1})^{m}.
\end{equation}

Consider the following iterative process of removing vertices from
$G_{ex}$. Put $G_0=\gex$ and $n_0=n$. Let $v_0\in V(G)$
be an arbitrarily chosen vertex of $G_0$ satisfying
$d(v_0)<(1-\frac{3}{3k-4})n_0$. Define $G_1=G-v_0$ and
$n_1=n_0-1$. For $j\geq 1$ if the minimum degree in $G_j$ satisfies 
$\delta(G_j)\geq(1-\frac{3}{3k-4})n_j$
then stop the process, otherwise take a vertex $v_j\in V(G_j)$ of
degree
$d_{G_j}(v_j)<(1-\frac{3}{3k-4})n_j$ and define
$G_{j+1}=G_j-v_j$ and $n_{j+1}=n_j-1$. 

We first show that this process must stop after at most $n/2$ steps. 
To see this note that the number
of copies of $K_m$ removed with each deleted vertex is 
exactly the number of
copies of $K_{m-1}$ in its neighborhood. By Proposition
\ref{prop:ex(n,K_m,H)} for any $(k-1)$-chromatic graph $H'$,
$ex(n,K_{m-1},H')=(1+o(1))\binom{k-2}{m-1}(\frac{n}{k-2})^{m-1}$. As
$\gex$ is $H$-free, the neighborhood of any vertex should be
$(H-v)$-free, where $v\in V(H)$ is such that $\chi(H-v)=k-1$. 

Thus at step $j$ (starting to count from $j=0$), at most
$(1+o(1))\binom{k-2}{m-1}(\frac{n_j(1-\frac{3}{3k-4})}{k-2})^{m-1}$
copies of $K_m$ have been removed.

As the following equality holds $$
\frac{1}{k-1}-(1-\frac{3}{3k-4})\frac{1}{k-2}=\frac{1}{(3k-4)(k-2)(k-1)}
$$ one can choose $\delta=\delta(k,m)>0$ so that
$((1-\frac{3}{3k-4})\frac{1}{k-2})^{m-1}=\frac{1}{(k-1)^{m-1}}(1-\delta)$. 
Thus the number of copies of $K_m$ removed at step $j$ 
is no more than: 
\begin{equation}
\label{eq:numOfK_mRemoved} 
(1+o(1)) n^{m-1}_j \binom{k-2}{m-1}\frac{1}{(k-1)^{m-1}}(1-\delta)
=(1+o(1)) n^{m-1}_j \binom{k-1}{m}\frac{m}{(k-1)^{m}}(1-\delta) 
\end{equation}

Together with the fact that 
$$
\sum_{r=0}^{n/2-1}
(n-r)^{m-1}
\leq (1+o(1))(\frac{1}{m}-\frac{1}{m}\frac{1}{2^m})n^m
$$ 
we conclude that 
the number of copies of $K_m$ removed during the first 
$\frac{n}{2}$ steps is at most
\begin{align*} 
&(1+o(1))\sum_{j=0}^{n/2-1} n^{m-1}_j
\binom{k-1}{m}\frac{m}{(k-1)^{m}}(1-\delta) \\\leq&(1+o(1))(1-\delta)
\binom{k-1}{m}\frac{1}{(k-1)^{m}}(1-\frac{1}{2^m})n^m 
\end{align*}

The graph $G_{n/2}$ has at most
$ex(\frac{1}{2}n,K_m,H)=(1+o(1))\binom{k-1}{m}
(\frac{1}{2}\frac{n}{(k-1)})^m$ 
copies of $K_m$, and hence the total number of copies of $K_m$ 
in $G_{ex}$ is at most 
\begin{align*} 
&(1+o(1))n^m(1-\delta) \binom{k-1}{m}\frac{1}{(k-1)^{m}}
(1-\frac{1}{2^m})+(1+o(1))n^m\binom{k-1}{m}\frac{1}{(k-1)^m}
\frac{1}{2^m} \nonumber \\ 
=&(1+o(1))n^m(1-\delta(1-\frac{1}{2^m}))
\binom{k-1}{m-1}\frac{1}{(k-1)^{m}} 
\end{align*} 
But if $\epsilon$ is small enough this contradics (\ref{e31}).
Thus the
process must stop after  $r+1 \leq \frac{n}{2}$ steps.

As $\delta(G_r)\geq (1-\frac{3}{3k-4})n_r$ and $H$ is edge
critical,
Corollary \ref{cor:ifFreeThenChrom} implies that $\chi(G_r)\leq k-1$.
Define $V(G_r)=V_r$. 

The $k-1$ partite subgraph of $G[V_r]$ with the maximum possible number of
copies of $K_m$ has at least as many copies of $K_m$ as $G_r$. By Lemma
\ref{lem:returningVertices}, part 1 we can now add the vertices
removed  during the steps of the process
starting from $j=r-1$ until $j=0$, keeping the resulting subgraph
$(k-1)$-partite, where with each such vertex
we add at least 
$(1+o(1)) n_j^{m-1}\binom{k-1}{m}\frac{m}{(k-1)^m}(1-c_1\epsilon)$
copies of $K_m$. Assuming that $\epsilon$ is small enough to ensure, say,
$c_1\epsilon<\delta/2$
 it follows that in 
each such step the number of
added copies of $K_m$ exceeds the number of copies removed in the
corresponding removal step.

When all the vertices are back we obtain a $k-1$ partite subgraph 
of $G$ containing
more copies of $K_m$ than $G_{ex}$. This subgraph is $H$-free,
contradicting the maximality of $G_{ex}$. Thus the inequality
$\delta(G_{ex}) \geq (1-\frac{3}{3k-4})n$ must hold and the desired 
result follows.
\end{proof}

\section{Maximizing the number of blow-ups of cliques}

To prove Theorem \ref{thm:mainBlowUp} we first need a good estimate on
$ex(n,K_m(t),H)$ for $H$ satisfying $\chi(H)=m+1$.

\begin{proposition} \label{prop:ex(n,K_m(t),H)}  For integers $m$ and $t$ and any 
fixed graph $H$ such
that $\chi(H)=m+1$,
$$ex(n,K_m(t),H)=(1+o(1))\binom{n/m}{t}^m $$ \end{proposition}

\begin{proof} To show that $ex(n,K_m(t),H)\geq(1+o(1))\binom{n/m}{t}^m$
it is enough to take the $m$-sided Tur\'an graph (i.e. the $m$-partite graph
with sides of nearly equal size). As $\chi(H)=m+1$ it is $H$-free and
has $(1+o(1))\binom{n/m}{t}^m$ copies of $K_m(t)$.

As for the upper bound, in \cite{ASh} it is shown that the graph which
is $K_{m+1}$ free and has the maximum possible number of copies of
$K_m(t)$ is a complete multipartite graph. It is not difficult to see
that the Tur\'an graph maximizes the number of copies of $K_m(t)$ among these. Thus
$$ ex(n,K_m(t),K_{m+1})=(1+o(1)) \binom{n/m}{t}^m $$

Let $H$ be as in the proposition, and let $G$ be an $H$-free graph on
$n$ vertices with the maximum number of copies of $K_m(t)$. By Lemma
\ref{lem:HfreeIsLikeK_k} $G$ can be made $K_{m+1}$-free by deleting
$o(n^2)$ edges, and with them at most $o(n^2)O(n^{mt-2})=o(n^{mt})$ copies of
$K_m(t)$. Let $G'$ be the graph obtained by removing those $o(n^2)$
edges. 

As $G'$ is $K_{m+1}$-free we get
\begin{align*}
(1+o(1))ex(n,K_m(t),H)&=(1+o(1))\N(G,K_m(T))=\\
&=\N(G',K_m(t))\leq ex(n,K_m(t),K_{m+1})=(1+o(1))
\binom{n/m}{t}^m
\end{align*}

and so $ex(n,K_m(t),H)\leq (1+o(1)) \binom{n/m}{t}^m$
as needed. \end{proof}

The idea of the proof of Theorem \ref{thm:mainBlowUp} is similar to the
one of Theorem \ref{thm:mainComplete} but some of the estimates are more
involved. 
\begin{proof}[Proof of Theorem \ref{thm:mainBlowUp}] Let $G$
be a graph with $\delta(G)>(1-\epsilon)n$ and let $G_{ex}\in
\mathcal{G}(K_m(t),H)$. 
If $\delta(\gex)\geq (1-\frac{3}{3m-1})n$ then  by Corollary
\ref{cor:ifFreeThenChrom},  $\chi(G_{ex})\leq m$, as $H$ is
edge critical and we are done.

Assume towards contradiction that $\delta(G_{ex})<(1-\frac{3}{3m-1})n$. 
Consider the following iterative process, similar to the one in the 
proof of Theorem \ref{thm:mainComplete}. Put $G_0=G_{ex}$ and $n_0=n$. 
At step $j>0$ if $G_j$ satisfies $\delta(G_j)\geq (1-\frac{3}{3m-1})n_j$ 
then stop the process, otherwise take $v_j\in V(G_j)$ of degree 
$d_{G_j}(v_j)\leq(1-\frac{3}{3m-1})n_j$ and define $G_{j+1}=G_j-v_j$ 
and $n_{j+1}=n_j-1$. We show that the
process must stop after at most $\frac{n}{2}$ steps.

To bound the number of $K_m(t)$ removed at each step we take care of two
cases. If in a copy of $K_m(t)$ there is a vertex in the same color
class of $v_i$ that is a neighbor of $v_i$ in $G_i$, call this copy
\textit{dense}. If all of the vertices in the color class of $v_i$ are
non-neighbors of it in $G_i$ call the copy \textit{sparse}.

First we  estimate the number of dense copies. Let $K^{+}_{m-1}(t)$ be
the graph obtained by taking $K_{m-1}(t)$ and adding to it a vertex that
is connected to all of the other vertices. The number of dense copies of
$K_m(t)$ containing $v_i$ is at most  $\N (G[N(v)],K^{+}_{m-1}(t))n_i^{t-2}$.

As $H$ is edge critical there is a vertex $v\in V(H)$ such that
$\chi(H-v)=m$, let $H'=H-v$. By a result in
$\cite{ASh}$ if $H$ is contained in a blow-up of $T$ then
$ex(n,T,H)=o(n^{v(T)})$. As the neighborhood of $v_i$ must be $H'$-free
and as $H'$ is contained in a blow-up of $K_{m-1}^{+}(t)$, it follows that
$\N(G[N(v)],K^{+}_{m-1}(t))=o(|N(v)|^{t(m-1)+1})$. Thus the number of dense
copies of $K_m(t)$ in $G_{j+1}$ containing $v_i$  is $o(n_i^{tm-1})$.

As for the sparse copies, let $A(v_i)$ be the number of sparse copies of
$K_m(t)$ in $G_i$ containing $v_i$. Let
$N^{c}_{G_i}(v_i)=V(G_i)\setminus( N_{G_i}(v_i)\cup \{v_i\})$ and $d=d_{G_i}(v_i)$,
and let $H'=H-v$ for $v\in V(H)$ such that $\chi(H')=m$. 
Using Proposition \ref{prop:ex(n,K_m(t),H)} we obtain the following bound on the 
number of sparse copies of $K_m(t)$ containing $v_i$
\begin{align*} A(v_i)\leq \sum_{{u_1,\dots,u_{t-1}}\subseteq
N^{c}_{G_i}(v_i)}& ex\big(|N(v_i)\cap N(u_1)\cap \dots \cap
N(u_{t-1})|,K_{m-1}(t),H'\big)\\ \leq &
\binom{n_i-d-1}{t-1}ex(d,K_{m-1}(t),H')\\ =&(1+o(1)) \binom{n_i-d}{t-1}
\binom{d/(m-1)}{t}^{m-1}\\ \leq& (1+o(1))\frac{{(n_i-d)}^{t-1}}{(t-1)!}
(\frac{{d}^{t}}{(m-1)^t t!})^{m-1} \end{align*}

To bound this quantity consider the following function
$f(d)=d^{t(m-1)}(n_i-d)^{t-1}$.  Note that $f(d)$ is a polynomial in $d$, $f(d)> 0$ 
for $0<d<n_i$ and $f(n_i)=f(0)=0$. Furthermore 
\begin{align*} f'(d)=&(d^{t(m-1)}(n_i-d)^{t-1})'\\
%
=&d^{t(m-1)-1}(n_i-d)^{t-2}[t(m-1)(n_i-d)-(t-1)d] \end{align*} 
Thus
$f'(d)=0$ for $d=0$, $d=n_i$,
$d=(\frac{t(m-1)}{t(m-1)+(t-1)})n_i=(1-\frac{t-1}{tm-1})n_i=:\beta$ 
and is positive in $[0,\beta]$. It follows that between 
$0$ and $n$ $f(d)$ obtains its global  maximum at the
single values $0<\beta<n$ for which 
$f'(\beta)=0$, and it is increasing in $[0,\beta]$.

In our case $d<(1-\frac{3}{3m-1})n_i$ and as $m>2$ it follows that
$1-\frac{t-1}{tm-1}>1-\frac{3}{3m-1}$. 
We conclude that $f(d)\leq f((1-\frac{3}{3m-1})n_i)$. Plugging this value it follows that
\begin{align*} A(v_i)\leq& 
(1+o(1)) \frac{{(\frac{3}{3m-1}n_i)}^{t-1}}{(t-1)!}
(\frac{{((1-\frac{3}{3m-1})n_i)}^{t}}{(m-1)^{t}t!})^{m-1}\\ = &(1+o(1))
n_i^{mt-1}(\frac{3}{3m-1})^{t-1}(1-\frac{3}{3m-1})^{t(m-1)}\frac{1}{(t-1)!(m-1)^{t(m-1)}(t!)^{m-1}} \end{align*}

Next we  bound
$(\frac{3}{3m-1})^{t-1}(1-\frac{3}{3m-1})^{t(m-1)}$. As
$\frac{3}{3m-1}=\frac{1}{(3m-1)m}+\frac{1}{m}$  the following holds:

\begin{align*}
(\frac{3}{3m-1})^{t-1}(1-\frac{3}{3m-1})^{t(m-1)}=&(\frac{1}{m}+\frac{1}{(3m-1)m})^{t-1}(\frac{m-1}{m}-\frac{1}{(3m-1)m})^{t(m-1)}\\ =&\frac{1}{m^{t-1}}(\frac{m-1}{m})^{t(m-1)}(1+\frac{1}{3m-1})^{t-1}(1-\frac{1}{(3m-1)(m-1)})^{t(m-1)} \\ \leq & \frac{1}{m^{t-1}}(\frac{m-1}{m})^{t(m-1)}(1+\frac{1}{3m-1})^{t-1}e^{-t/(3m-1)}\\ =&\frac{1}{m^{t-1}}(\frac{m-1}{m})^{t(m-1)}[(1+\frac{1}{3m-1})e^{-1/(3m-1)}]^{(t-1)}e^{-1/(3m-1)}\\ \leq &\frac{1}{m^{t-1}}(\frac{m-1}{m})^{t(m-1)}(1-\delta) \end{align*} 

For an appropriate $\delta:=\delta(m,d)>0$, indeed such a $\delta$ exists as $e^{-1/(3m-1)}<1$ and $(1+\frac{1}{3m-1})e^{-1/(3m-1)}<1$ for $m>2$.

Therefore, the number of copies of $K_m(t)$ (both dense and
sparse) removed  at step $i$ is at most
\begin{align}\label{eq:numOfK_m(t)Removed} 
(1+o(1))
n_i^{mt-1}&(1-\delta)\frac{1}{m^{t-1}}(\frac{m-1}{m})^{t(m-1)}
\frac{1}{(t-1)!(m-1)^{t(m-1)}(t!)^{m-1}} + o(n_i^{mt-1}) \nonumber\\ 
&=(1+o(1))n_i^{mt-1}(1-\delta)\frac{mt}{m^{tm} (t!)^{m}}	. \end{align}

If the process continued for $\frac{n}{2}$ steps, as
$\sum_{r=0}^{n/2}(n-r)^{mt-1}\leq (1+o(1))\frac{n^{mt}}{mt}(1-\frac{1}{2^{mt}})$
the total number of copies of $K_m(t)$ removed is at most 
\begin{align*}
&\sum_{r=0}^{(n/2)-1}(1+o(1)) (n-r)^{mt-1}(1-\delta)\frac{mt}{m^{tm}
(t!)^{m}}\\ \leq& (1+o(1))
n^{mt}(1-\delta)(1-\frac{1}{2^{mt}})\frac{1}{m^{tm} (t!)^{m}}.
\end{align*}

By proposition \ref{prop:ex(n,K_m(t),H)} in the graph $G_{n/2}$ the number of 
copies of $K_m(t)$ is at most
$$
(1+o(1))\binom{n/(2m)}{t}^{m}\leq(1+o(1))
(\frac{(n/2)^t}{m^{t}t!})^m=(1+o(1))n^{tm}\frac{1}{m^{tm}(t!)^m}
\frac{1}{2^{mt}}.
$$ 
Thus the number of copies of $K_m(t)$ in $G_{ex}$ is at most
$$n^{tm}\big((1-\delta(1-\frac{1}{2^{mt}}))\frac{1}{m^{tm}(t!)^m}\big)
$$ in contradiction to the maximality of $G_{ex}$. And so the process
must stop after $n/2$ steps.

Assume that we have stopped at step $r<n/2$ and let $V_r=V(G_r)$. 
By Corollary \ref{cor:ifFreeThenChrom}
$G_r$ must be $m$-partite, thus the $m$-partite subgraph of $G[V_r]$ 
with the
maximum possible number of copies of $K_m(t)$ has at least as many copies 
of $K_m(t)$ as $G_r$. 

By Lemma \ref{lem:returningVertices}, part 2, we can return the vertices 
removed in the process in a reverse order (starting from $v_{r-1}$ until 
$v_0$) keeping the graph $m$-partite and adding with each vertex $v_j$ at least $(1+o(1))n_j^{mt-1}\frac{mt}{(t!m^t)^m}(1-c_2\epsilon)$ copies of $K_m(t)$. 
Assuming that $\epsilon$ is small enough to ensure, say, $c_2\epsilon<\delta/2$ 
it follows that with each vertex $v_j$ we add more copies of $K_m(t)$ than were 
removed at the corresponding step. 

Thus when all vertices are returned we obtain an $m$-partite graph with 
more copies of $K_m(t)$ than $G_{ex}$. As an $m$-partite graph is
$H$-free this contradicts the definition of $G_{ex}$. Thus it must be
that $\delta(G_{ex})\geq (1-\frac{3}{3m-1})n$ and $G_{ex}$ is
$m$-partite. 
\end{proof}

\section{Forbidding graphs that are not edge critical}

The proofs of Theorems \ref{thm:mainComplete} and \ref{thm:mainBlowUp}
actually give a stronger result than stated, as follows
\begin{lemma} \label{cor:BigDiff} Let $m<k$ and $t$ be
integers, let $G$ be a graph on $n$ vertices such that
$\delta(G)>(1-\epsilon)n$, where $\epsilon:=\epsilon(k,m,t)>0$
is sufficiently small, and let $G_{ex}\in
\mathcal{G}_{ex}(K_m(t),K_k)$.
Assume that $k=m+1$ or $t=1$. Then for every $K_k$-free subgraph of $G$
on the same set of vertices, say $G^1\subseteq G$, at least one of the
following holds: \begin{enumerate} \item $\N(G^1,K_m(t))\leq \gamma
\N(G_{ex},K_m(t))$ for some $\gamma:=\gamma(k,m,t)<1$. \item $G^1$ can
be made $k-1$-chromatic by deleting $o(n^2)$ edges. \end{enumerate}
\end{lemma}

\begin{proof} If $\delta(G^1)\geq (1-\frac{3}{3k-4})v(G^1)$ then by
Theorem
\ref{thm:highDegforKk} $G^1$ is $k-1$-chromatic and hence case (2) 
holds and we are done. If $\delta(G^1)<
(1-\frac{3}{3k-4})v(G^1)$  we consider a similar process to the one in
the proofs
of Theorem \ref{thm:mainComplete} and \ref{thm:mainBlowUp}. For step
$j=0$ of the process define
$G^1_0=G^1$, for steps $j>0$ let $v_j$ be a vertex of minimum degree in
$G^1_{j-1}$ and define $G^1_j=G^1_{j-1}-v_j$. The process stops when
either
 $\delta(G^1_j)\geq (1-\frac{3}{3k-4})v(G^1_j)$
or when $v(G^1_j)=\alpha n$ for $\alpha:=\alpha(\gamma)$ small enough.

The calculations in the proofs of Theorems \ref{thm:mainComplete} and
\ref{thm:mainBlowUp} (see equation (\ref{eq:numOfK_mRemoved}) and
(\ref{eq:numOfK_m(t)Removed})) yield that when removing a vertex of
degree less than $ (1-\frac{3}{3k-4})v(G^1_j)$ the number of copies of
$K_m(t)$
removed with it is at most $$ (1+o(1))(1-\delta)v(G^1_j)^{mt-1}
\binom{k-1}{m}\frac{mt}{(k-1)^m (t!m^{t-1})^m}. $$

Assume that the process stops at step $r$. If $r=o(n)$ then by Theorem
\ref{thm:highDegforKk} the graph $G^1_r$ is $k-1$-chromatic.
In these $r$ steps $o(n)$ vertices were deleted and with them no more
than
$o(n^2)$ edges, thus case (2) holds.

If $r=cn$ for some $c\leq 1-\alpha$ define the graph $G_r$ as follows. 
If $c=1-\alpha$ take $G_r$ to be a
$k-1$-chromatic subgraph of $G^1$ on the vertices of $G_r^1$ with the
maximum possible
number of copies of $K_m(t)$.
It must be that
$\N(G_r,K_m(t))-\N(G^1_r,K_m(t))<\alpha'n^{mt}$ for an appropriate
$\alpha'=\alpha'(\alpha)$ which tends to $0$ as $\alpha$ tends to $0$.
If $c>\alpha$ take $G_r=G^1_r$, by Theorem \ref{thm:highDegforKk} this
graph is $k-1$-chromatic.

As $G_r$ is $k-1$-chromatic in both cases we can apply Lemma
\ref{lem:returningVertices} to it and  add back
the vertices removed in the process, starting from $j=r-1$ to $j=1$,
while keeping the graph $k-1$-chromatic. We get that the number of
copies of $K_m(t)$ added with each
vertex $v_j$ is at least $$(1+o(1))(1-c\epsilon)v(G^1_j)^{mt-1}
\binom{k-1}{m}\frac{mt}{(k-1)^m (t!m^{t-1})^m}.$$ 
Let $G^2$ be the graph obtained after adding back all the vertices.

Assume that $\epsilon$ is small
enough to ensure that $\delta - c\epsilon>c'>0$ for some
$c':=c'(\gamma)$. Let $n_j=v(G^1_j)=n-j$, and note
that $\sum_{j=0}^r n_j^{mt-1}=\sum_{j=0}^{cn} (n-j)^{mt-1}\geq
(1+o(1))n^{mt}\frac{1}{mt}(1-(1-c)^{mt})$. Thus the difference in the
number of copies of $K_m(t)$ in $G^1$ and $G^2$ is at least
\begin{align*} \N(G^2,K_m(t))-\N(G^1,K_m(t))\geq &(1+o(1)) \sum_{j=0}^r
c'n_j^{mt-1} \binom{k-1}{m}\frac{mt}{(k-1)^m
(t!m^{t-1})^m}-\alpha'n^{mt}
\\ 
\geq &(1+o(1))
(c'-\alpha')(1-(1-c)^{mt})n^{mt}\binom{k-1}{m}\frac{1}{(k-1)^m
(t!m^{t-1})^m}
\\
 =&(1+o(1))(1-\gamma)\N(G_{ex},K_m(t)) \end{align*}
where $c'$ and $\alpha'$ are chosen so that the last equality holds. 

As $G^2$ is a $K_k$-free subgraph of $G$,  $ \N(G_{ex},K_m(t))\geq
\N(G^2,K_m(t)) $ and
thus $\gamma\N(G_{ex},K_m(t))\geq \N(G^1,K_m(t))$  and case (1) holds,
as needed.

\end{proof}

The proof of Proposition \ref{prop:otherH} is now a simple corollary of
the last lemma.
\begin{proof}[Proof of Proposition \ref{prop:otherH}.] Let $G$ be a
graph on $n$ vertices with $\delta(G)>(1-\epsilon)n$ and let
$G_{ex(H)}\in \mathcal{G}_{ex}(K_m(t),H)$. By Lemma
\ref{lem:counting} $\N(G_{ex(H)},K_m(t))=\Theta(n^{mt})$.
%
By Lemma \ref{lem:HfreeIsLikeK_k} there is a graph $G_1\subseteq
G_{ex(H)}$ which is $K_k$-free and $e(G_{ex(H)})-e(G_1)=o(n^2)$, and
thus $$\N(G_{ex(H)},K_m(t))=(1+o(1))\N(G_1,K_m(t)).$$

Let $G_{ex(K_k)}\in \mathcal{G}(K_m(t),K_k)$.
%
To apply Lemma \ref{cor:BigDiff} we show that
\begin{equation}\label{eq:Case1NotTrue}
\N(G_1,K_m
(t))=(1+o(1))\N(G_{ex(K_k)},K_m(t)).
\end{equation}
By theorems
\ref{thm:mainComplete} and \ref{thm:mainBlowUp} 
$G_{ex(K_k)}$ is $k-1$-chromatic, and so it is $H$-free. Together with the fact that
 $G_1$ is $K_k$-free, we get
 $$\N(G_1,K_m(t))\leq
\N(G_{ex(K_k)},K_m(t))\leq
\N(G_{ex(H)},K_m(t))=(1+o(1))\N(G_1,K_m(t))$$ 
implying (\ref{eq:Case1NotTrue}).

Thus case (1) in Lemma \ref{cor:BigDiff} does not hold for $G_1$, and so case 
(2) must hold, i.e. $G_1$ can be made $k-1$-chromatic by deleting $o(n^2)$ edges. 
 As we got $G_1$ from
$G_{ex(H)}$ by deleting $o(n^2)$ edges we get the required result.
\end{proof}

\section{Concluding remarks and open problems} 
\begin{itemize} 
\item Corollary \ref{cor:ifFreeThenChrom} and Theorems
\ref{thm:mainComplete} and \ref{thm:mainBlowUp} cannot be directly
generalized for graphs 
$H$ which are not edge critical. 
In \cite{AS} the following is
shown  (a weaker version of this statement is proved in \cite{AShS})
\begin{theorem}[\cite{AS}] 
\label{thm:deleteEdgeForChrom} Let $H$ be a
fixed graph on $h$ vertices such that $\chi(H)=k\geq 3$ and let $G$ be
an $H$-free graph on $n$ vertices with $\delta(G)\geq
(1-\frac{3}{3k-4}+o(1))n$, where $n$ is large enough. Then one can
delete at most $O(n^{2-1/(4(k-1)^{2/3}h)})$ edges from $G$ and make it
$k-1$-colorable. 
\end{theorem}

This suggests that 
a stronger version of Proposition \ref{prop:otherH},
stating that any extremal graph $G_{ex}$ as in the proposition 
can be made $k-1$-chromatic by
deleting $O(n^{2-\mu(H)})$ edges for some $\mu(H)>0$, 
is likely to be true.

\item Theorem \ref{thm:mainBlowUp} is limited to the case where
$\chi(H)=m+1$, and from this we also get the condition in Proposition
\ref{prop:otherH}. One of the problems in extending it to graphs $H$
with higher chromatic number is that of finding an explicit tight
bound on
$ex(n,K_m(t),K_k)$ for $k>m+1$. In \cite{ASh} it is shown that the
extremal graph is $k-1$-partite. However, it is not difficult to check that for $k\geq m+2$ such that
$m\nmid k-1$ and large values of $t$,
the parts are not of equal sizes.
%

\item Theorems in the same spirit as those proven here
may hold for other pairs of graphs $T$
and $H$. In \cite{ASh} it is observed that if $H$ is not contained in any
blow-up of $T$ then $ex(n,T,H)=\Theta(n^{v(T)})$. This of course does
not mean that the extremal graph is a blow-up of $T$, but in cases 
it is a similar behavior 
to that in the results proven here might be expected.

A notable example is the case $T=C_5$ and $H=K_3$. In \cite{HHKNR}
and independently \cite{G} it is shown that when $5|n$ the extremal
graph is the equal sided blow-up of $C_5$. It might be true that
this behavior holds for subgraphs of graphs of high minimum degree and
not only for subgraphs of $K_n$, that is, the extremal subgraphs
in this case 
may be subgraphs of the equal sided blow-up of $C_5$.

\item
The problem of obtaining the best possible bounds for the 
minimum degree
ensuring that the results stated in Theorems 
\ref{thm:mainComplete}, \ref{thm:mainBlowUp} and Propositon
\ref{prop:otherH} hold is also interesting, but appears to be
difficult. Even the very special case of
Theorem \ref{thm:mainComplete} with
$H=K_3$ and $m=2$, conjectured in
\cite{BKS} to be $3/4+o(1)$, is open.
\end{itemize}


\begin{thebibliography}{99}

\bibitem{Al}
N. Alon, Problems and results in extremal combinatorics,
II, Discrete Math. 308, 4460-4472, (2008).

\bibitem{AFKS} N. Alon, E. Fischer, M. Krivelevich and M. Szegedy,
Efficient testing of large graphs. Combinatorica 20(4) , 451-476,
(2000).

\bibitem{AShS} N. Alon, A. Shapira and B. Sudakov, Additive
approximation for edge-deletion problems, Annals of mathematics ,
371-411, (2009).

\bibitem{ASh} N. Alon and C. Shikhelman, Many $T$ copies in $H$-free
graph, Journal of Combinatorial Theory, Series B,  121, 146-172, (2016).


\bibitem{AS} N. Alon, and B. Sudakov, $H$-free graphs of large minimum
degree, Electronic Journal of Combinatorics, (2006).

\bibitem{AES} B. Andr\'asfai, P. Erd\H{o}s and V. T. S\'os, On the
connection between chromatic number, maximal clique and minimal degree
of a graph, Discrete Mathematics 8.3, 205-218, (1974).

\bibitem{BKS} J. Balogh, P. Keevash and B. Sudakov, On the minimal
degree implying equality of the largest triangle-free and bipartite
subgraphs, Journal of Combinatorial Theory, Series B 96, no. 6, 919-932,
(2006).

\bibitem{BSTT} A. Bondy, J. Shen, S. Thomass\'e and C. Thomassen,
Density conditions for triangles in multipartite graphs, Combinatorica
26, no. 2, 121-131, (2006).

\bibitem{CF} D. Conlon and J. Fox, Graph removal lemmas, Surveys in
combinatorics 1.2, 3-50, (2013).

\bibitem{E} P. Erd\H{o}s, On extremal problems of graphs and generalized
graphs, Israel Journal of Mathematics 2.3, 183-190, (1964).

\bibitem{E83} P. Erd\H{o}s, On some problems in graph theory,
combinatorial analysis and combinatorial number theory, Graph Theory and
Combinatorics (Cambridge, 1983), Academic Press, London, 1-17, (1984).

\bibitem{ES66} P. Erd\H{o}s and M. Simonovits, A limit theorem in graph
theory, Studia Sci. Math. Hungar 1, (1966).

\bibitem{ES} P. Erd\H{o}s and M. Simonovits, On a valence problem in
extremal graph theory, Discrete Mathematics 5.4, 323-334, (1973).

\bibitem{ESt} P. Erd\H{o}s and A. Stone, On the structure of linear
graphs, Bull. Amer. Math. Soc. 52, 1087-1091, (1946).

\bibitem{HHKNR} H. Hatami, J. Hladk\'y, D. Kr\'al', S. Norine and A.
Razborov, On the number of pentagons in triangle-free graphs, J. Combin.
Theory Ser. A, 120.3, 722--732 (2013).

\bibitem{Fo} J. Fox, A new proof of the graph removal lemma, Annals of
Mathematics 174,  561-579, (2011).

\bibitem{G} A. Grzesik, On the maximum number of five-cycles in a
triangle-free graph. Journal of Combinatorial Theory, Series B, 102(5),
1061-1066, (2012).


\bibitem{RSz} I. Z. Ruzsa and E. Szemer\'edi, Triple systems with no six
points carrying three triangles, in Combinatorics, Coll. Math. Soc. J.
Bolyai 18, Volume II, 939--945, (1976).

\bibitem{T} P. Tur\'an, On an extremal problem in graph theory,
Mat. Fiz. Lapok,  48, 436-452, (1941).
\end{thebibliography}
\end{document}